\documentclass{amsart}
\usepackage{amssymb}
\usepackage{mathrsfs}
\usepackage{stmaryrd}
\usepackage{bbm}
\usepackage{color}
\usepackage{oldgerm}
\usepackage[english]{babel}
\usepackage[T1]{fontenc}
\usepackage[latin1]{inputenc}
\usepackage[all]{xy}
\usepackage{hyperref}
\usepackage[all]{xy}
\usepackage{extarrows}

\newtheorem{theorem}{Theorem}[section]

\newtheorem{lemma}[theorem]{Lemma}

\newtheorem{conjecture}{Conjecture}[section]

\theoremstyle{definition}

\numberwithin{equation}{section}

\usepackage{amsmath}
\usepackage{amsthm}
\usepackage{mathrsfs}
\usepackage{amsfonts}

\begin{document}

\title {Unit roots of the unit root $L$-functions}

\author[L.P. Yang]{Liping Yang}
\address{School of Mathematical Sciences, Chengdu University of Technology, Chengdu 610059, P.R. China}
\email{yanglp2013@126.com}
\author[H.Zhang]{Hao Zhang}
\address{Shing-Tung Yau Center of Southeast University, Yifu Architecture Building, No.2, Sipailou, Nanjing 210096, P.R. China}
\email{zhanhgao@126.com}
\date{}
\maketitle

\begin{abstract}
Adolphson and Sperber characterized the unique unit root of $L$-function associated with toric exponential sums in terms of the $\mathcal{A}$-hypergeometric functions. For the unit root $L$-function associated with a family of toric exponential sums, Haessig and Sperber conjectured its unit root behaves similarly to the classical case studied by Adolphson and Sperber. Under the assumption of a lower deformation hypothesis, Haessig and Sperber proved this conjecture. In this paper, we demonstrate that Haessig and Sperber's conjecture holds in general.

\medskip
\noindent {\bf Key words:} unit root, $L$-function, Dwork trace formula, hypergeometric function.

\medskip
\noindent {\bf Mathematics Subject Classification:} 11T23, 11S40.

\end{abstract}

\section{Introduction}

Let $\mathbb F_q$ be the finite field of $q=p^a$ elements with characteristic $p$ and let $\bar{\mathbb{F}}_q$ be its algebraic closure.
Let
$$f(\Lambda,X)=\sum a_{r,u}\Lambda^{r}X^u\in \mathbb{F}_q[\Lambda_1^{\pm},\ldots,\Lambda_s^{\pm},X_1^{\pm},\ldots,X_n^{\pm}]$$ be a Laurent polynomial.
 Fix a primitive $p$-th root of unity $\zeta_p$.
 For each $\lambda\in (\bar{\mathbb{F}}_q^*)^{s}$ and positive integer $m$, define the toric exponential sum
$$S_m\big(f(\lambda,X)\big):=\sum_{X\in (\mathbb{F}_{q^{md(\lambda)}}^*)^n} \zeta_p^{{\rm Tr}_{\mathbb{F}_{q^{md(\lambda)}}/\mathbb{F}_q}(f(\lambda,X))},$$
where $d(\lambda)$ is the degree $[\mathbb{F}_q(\lambda): \mathbb{F}_q]$ and $\mathbb{F}_q(\lambda)$ is the  field extension  by adjoining every coordinate of $\lambda$ to $\mathbb{F}_q$.
It is well-known that the associated $L$-function
$$L\big(f(\lambda,X),T\big):=\exp\Big(\sum_{m\ge 1}\frac{S_m(f(\lambda,X))T^m}{m}\Big)$$
is a rational function and its roots and poles are algebraic integers. Much work has been done on the $p$-adic and absolute values of the roots and poles(see \cite{AS,De80,De91, Wan93}), which is closely related to the Weil conjecture.

Another direction is to express the roots and poles of $L$-functions in terms of special values of $\mathcal{A}$-hypergeometric functions. The $p$-adic hypergeometric differential equations arise because they describe the variation of $p$-adic cohomology of a parametrized family of exponential sums. Dwork \cite{Dw69} expressed the unit root of a non-supersingular elliptic curve $y^2=x(x-1)(x-\lambda)$ in terms of the Gaussian hypergeometric function $F(\frac{1}{2},\frac{1}{2},1;\lambda)$.
This analysis motivated Dwork's general study of $p$-adic periods. For the unit root of the family of Kloosterman sums $x+\lambda/x$, Dwork also established similar results using $p$-adic Bessel function \cite{Dw74}. Sperber\cite{SP0} generalized Dwork's work to the family of the $n$-dimensional Kloosterman sums.
Dwork and Loeser \cite{DL93} established a systematic correspondence between the families of varieties or exponential sums and hypergeometric equations.
When $f(\Lambda,X)$ is of the form $\sum\Lambda_u X^u$, Adolphson and Sperber \cite{AS2} presented an explicit $p$-adic analytic formula for the unique unit root of $L$-function $L(f(\lambda,X),T)$  in terms of $\mathcal{A}$-hypergeometric functions.

For each $\lambda\in (\bar{\mathbb{F}}_q^*)^s$, Adolphson and Sperber proved that $L(f(\lambda,X),T)$ has a unique unit root, denoted by $\pi_0(\lambda)$, which is a 1-unit. Let $\mathbb Q_p$ be the $p$-adic number field and $\mathbb{Z}_p$ be the ring of $p$-adic integers. The unit root $L$-function for this family is defined as
$$L_{unit}(f,\kappa,T):=\prod_{\lambda\in |\mathbb {G}_m^s/\mathbb{F}_q|}\frac{1}{1-\pi_{0}(\lambda)^{\kappa}T^{d(\lambda)}},$$
where $\kappa\in \mathbb{Z}_p$.
Dwork \cite{Dw73} conjectured that for each $\kappa\in \mathbb{Z}_p$, the unit root $L$-function $L_{unit}(f,\kappa,T)$ is $p$-adic meromorphic.
Wan \cite{Wan99,Wan00a,Wan00b} proved that the Dwork's conjecture holds. In other words, the unit root $L$-function $L_{unit}(f,\kappa,T)$ is a quotient of $p$-adic entire functions:
$$L_{unit}(f,\kappa,T)=\frac{\prod_{i=1}^{\infty}(1-\alpha_i(\kappa)T)}{\prod_{j=1}^{\infty}(1-\beta_j(\kappa)T)} $$
with $\alpha_i(\kappa),\beta_j(\kappa)\rightarrow 0$ as $i,j\rightarrow \infty$.
Little is known about the zeros and poles.
Using the $p$-adic hypergeometric functions, Haessig and Sperber \cite{HS17} also provided an explicit expression of the unit root of $L_{unit}(f,\kappa,T)$  when $f(\Lambda,X)$ satisfies a lower deformation hypothesis.

To describe the form of the unit root, we introduce the following notations. Let $F(Y,\Lambda,X)=\sum Y_{r,u}\Lambda^{r}X^u$ be a new polynomial obtained by replacing the coefficients of $f(\Lambda,X)$ with new variables $Y_{r,u}$. Let $\bar{\mathbb Q}_p$ be the algebraic closure of $\mathbb{Q}_p$. Choose $\pi\in \bar{\mathbb{Q}}_p$ satisfying $\pi^{p-1}=-p$ and write $$\exp \pi F(Y,\Lambda,X)=\sum g_{r,u}(Y)\Lambda^rX^u.$$ Let $$G(Y):=g_{0,0}(Y)/g_{0,0}(Y^p).$$
Based on Adolphson and Sperber's results, which show that  $g_{0,0}(Y)$ satisfies the $A$-hypergeometric system and $G(Y)$ converges on the closed polydisk $|Y_{r,u}|_p\le 1$, Haessig and Sperber \cite{HS17} conjectured that the unit root has the following form.
\begin{conjecture}\label{conj}
Let $\hat{a}_{r,u}$ be the Teichm\"{u}ller lift of $a_{r,u}$. Let $\hat{a}=(\hat{a}_{r,u})$ and $G(\hat{a})=G(Y)|_{Y=\hat{a}}$. The unit root $L$-function $L_{unit}(f,\kappa,T)$ has a unique unit root, which is given by
$$\big(G(\hat{a})G(\hat{a}^p)\cdots G(\hat{a}^{p^{a-1}})\big)^{\kappa}.$$
\end{conjecture}
Haessig and Sperber\cite{HS17} proved this conjecture holds when $f(\Lambda,X)$ satisfies the lower deformation hypothesis, which is, if we write $f(\Lambda,X)=g(X)+h(\Lambda,X)$,
the $X$-support of $h(\Lambda,X)$ strictly lies in the Newton polyhedron of $X$-support of $g(X)$. Note that the Kloosterman case does not satisfy the lower deformation condition.

To avoid the lower deformation hypothesis, we define the $p$-adic Banach spaces on which the Frobenius maps act through absolute values rather than weight functions. However, it remains unknown whether the Dwork trace formulas hold in these new $p$-adic Banach spaces.
In this paper, one of our main work is to prove the Dwork trace formulas still hold in these new spaces, including Dwork's dual theory.
Finally, we will show that the conjecture holds in general.

\begin{theorem}\label{thm24}
For a Laurent polynomial $f(\Lambda,X)=\sum a_{r,u}\Lambda^{r}X^u\in \mathbb{F}_q[\Lambda_1^{\pm},\ldots,\Lambda_s^{\pm},X_1^{\pm},\ldots,X_n^{\pm}]$, Conjecture \ref{conj} holds.
\end{theorem}

By comparing Theorem 1.3 in \cite{AS2} and Theorem \ref{thm24}, we observe that the unit root of the unit root $L$-function behaves similarly to the unit root of the family of exponential sums. In fact, this result generalizes the classical case of the $L$-function for toric exponential sums.
Moreover, Haessig \cite{HS17} showed that the unit root $L$-function of one-dimensional Kloosterman family has a unit root $T=1$. Similarly, the authors \cite{YZ23} demonstrated that $T=1$ is the unit root of the unit root $L$-function for the $n$-dimensional Kloosterman family. We can conclude the same results from Theorem \ref{thm24} for the Kloosterman case.

This paper is organized as follows. In Section 2, we introduce the Dwork trace formula for the toric exponential sum. In Section 3, we prove the unit root $L$-function has one unique unit root by proving the Dwork trace formula for the infinity symmetric power $L$-function. In Section 4, we
introduce the dual theory for the infinity symmetric power $L$-function. Finally, in Section 5, we prove our main result, Theorem \ref{thm24}.

\section{The exponential sums}

In this section, we prove the Dwork trace formulas for exponential sums.
Consider the polynomial $f(X)=\sum_{ u\in A}\lambda_{u} X^{ {u}}\in \mathbb F_q[X_1^{\pm },\cdots,X_n^{\pm }]$, where $A$ is a finite subset of $\mathbb Z^n$ and $X^{ {u}}=X_1^{u_1}\cdots X_n^{u_n}$ for $ u=(u_1,\cdots,u_n)$.
For a positive integer $m$, the exponential sum associated to $f(X)$ is defined by
$$S_m\big(f(X)\big)= \sum_{X\in (\mathbb F_{q^m}^\ast)^n} \zeta_p^{\mathrm{Tr}_{\mathbb F_{q^m}/\mathbb F_q}(f(X))}.$$
The generating $L$-function associated with $S_m(f(X))$ is defined by
$$L\big(f(X),T\big)=\exp\Big(\sum_{m=1}^\infty S_m\big(f(X)\big)\frac{T^m}{m}\Big).$$

Let $\mathbb Q_q$ be the unramified extension of $\mathbb Q_p$ with residue field $\mathbb F_q$. Let $\delta_A$ be the cone in $\mathbb{R}^n$ generated by $A$.
Let $\omega=\max\{\vert u\vert: u\in A\}$. Choose  $\tilde{\pi}\in \overline{\mathbb Q}_p$ such that $0<\text{ord}_p\tilde{\pi}\le \frac{p-1}{p^2\omega}$. Recall that $\pi\in \overline{\mathbb Q}_p$ satisfies $\pi^{p-1}=-p$.
 Let $K=\mathbb Q_q(\tilde{\pi},\pi)$ and $R$ be the integer ring of $K$.

Define the following $p$-adic space
$$\mathcal C_0=\big\{\sum_{ u\in \delta_A\cap\mathbb Z^n}a_{ u}\tilde{\pi}^{| u|}X^{ u} : a_u\in K, \vert a_{ u}\vert \to 0 \text{ as }  u\to \infty\big\}.$$
It is a Banach space with respect to the norm
$$\Vert\sum_{ u\in \delta_A\cap\mathbb Z^n} a_{ u}\tilde{\pi}^{| u|}X^{ u} \Vert=\sup_{ u\in \delta_A\cap\mathbb Z^n} \vert a_{ u} \vert.$$
For $s\in \mathbb{R}$ and $s>1$, define
$$L(s):=\big\{\sum_{ u\in \delta_A\cap\mathbb Z^n}a_{ u}X^{ u} : a_u\in K, \vert a_{ u}\vert s^{\vert u\vert}\text{ is bounded} \big\}.$$
It is a Banach space with respect to the norm
$$\Vert\sum_{ u\in \delta_A\cap\mathbb Z^n} a_{ u}X^{ u} \Vert=\sup_{u\in \delta_A\cap\mathbb Z^n} \vert a_{ u} \vert s^{\vert u\vert}.$$
For $s=p^{\text{ord}_p\tilde{\pi}}$ and $s'>s$, we have
$$L(s')\subset \mathcal{C}_0 \subset L(s).$$

Let $\Theta(T)$ be the Dwork splitting function defined by $$\Theta(T):=\exp \pi(T-T^p)=\sum_{i=0}^{\infty}\theta_iT^i.$$
It is well-known that ${\rm ord}_p \theta_i\ge \frac{(p-1)i}{p^2}$ for all $i\ge 0$.
Let $\hat{f}(X)=\sum_{ u\in A}\hat\lambda_{u} X^{ {u}}\in K[X^{\pm}]$ be the Techim\"{u}ller lift of $f(X)$ and $H(X):=\exp\pi \big(\hat{f}(X)-\hat{f}(X^p)\big)$.

\begin{lemma}\label{lem01}
We have that
$$H(X)\in L(p^{\frac{p-1}{p^2\omega}}).$$
\end{lemma}
\begin{proof}
Write $$H(X)=\prod_{u\in A}\Theta(\hat{\lambda}_uX^u)=\prod_{u\in A}\sum_{i_u=0}^{\infty}\theta_{i_u}(\hat{\lambda}_uX^u)^{i_u}.$$
If we write $H(X)=
\sum_{v\in \delta_A\cap \mathbb{Z}^n}b_{v}X^{v},$
then $b_{v}=\sum_{(i_u)\in S_v}\prod_{u\in A}\theta_{i_u}\hat{\lambda}_u^{i_u}, $
where
\begin{equation}\label{eq3}
S_v=\big\{(i_u)_{u\in A}:\sum_{u\in A}u\cdot i_u=v\big\}.
\end{equation}
Hence
$${\rm ord}_pb_{v}\ge \min_{(i_u)\in S_v}\frac{p-1}{p^2}\sum_{u\in A} |i_{u}|\ge \min_{(i_u)\in S_v}\frac{(p-1)\big(\sum_{u\in A} |i_{u}||u|\big)}{p^2\omega}.$$
It follows that
$$ {\rm ord}_pb_{v}\ge \frac{(p-1)|v|}{p^{2}\omega}.$$
Hence $H(X)\in L(p^{\frac{p-1}{p^2\omega}})$. Then we prove Lemma \ref{lem01}.
\end{proof}

Consider the map defined by
$$\psi\big(\sum_{v}a_{v}X^{v}\big)=\sum_{v}a_{p v}X^{ v}.$$
One checks that $\psi(L(s))\subset L(s^p)$.

 Now we define the Frobenius map $\alpha_1$ and $\alpha_a$ acting on $\mathcal{C}_0$ by $\alpha_1=\psi\circ H(X)$ and $\alpha_a=\alpha_1^a$. Since $0<\text{ord}_p\tilde{\pi}\le \frac{p-1}{p^2\omega}$, we have $H(X)\in L(p^{\frac{p-1}{p^2\omega}})\subset L(p^{\text{ord}_p\tilde{\pi}})$. Hence The map $\alpha_1$ is a composition of the following maps:

$$\mathcal C_0\hookrightarrow L(p^{\text{ord}_p\tilde{\pi}})\stackrel{H(X)}\to  L(p^{\text{ord}_p\tilde{\pi}})\stackrel{\psi}\to L(p^{p\cdot\text{ord}_p\tilde{\pi}})\hookrightarrow \mathcal C_0.$$

One verifies the last map is completely continuous. Hence the Frobenius map $\alpha_1$ is completely continuous, and so is $\alpha_a$. We could also verify this fact directly.
It follows from Lemma \ref{lem01} that we can write $H(X)=\sum_{ v}b_vX^{ v}$ with
$\vert b_{ v}\vert  p^{\vert v\vert\text{ord}_p\tilde{\pi}}\le 1$.We have

$$\alpha_1\big(\tilde\pi^{\vert  u\vert} X^{ u}\big)=\psi\big(\sum b_{ v}\tilde\pi^{\vert  u\vert} X^{ u+ v}\big)=\psi\big(\sum b_{ v- u}\tilde\pi^{\vert  u\vert} X^{ v}\big)=\sum b_{p v- u}\tilde\pi^{\vert  u\vert-\vert  v\vert} \tilde\pi^{\vert  v\vert} X^{ v}.$$
and
$$\vert b_{p v- u} \vert\cdot \vert \tilde\pi \vert^{\vert  u\vert-\vert  v\vert}\le  p^{-\big(\vert p v- u\vert+\vert  u\vert-\vert  v\vert\big)\text{ord}_p\tilde{\pi}}\le  p^{-(p-1)\vert  v\vert\text{ord}_p\tilde{\pi}}.$$

Since the coefficient $b_{p v- u}\tilde\pi^{\vert  u\vert-\vert  v\vert}$ tends to $0$ as $\vert  v\vert \to 0$, which is independent of $ u$, we have $\alpha_1$ is a completely continuous operator of $\mathcal C_0$. It follows from \cite[Theorem 6.10]{M1} that $\alpha_1$ and $\alpha_a$ are nuclear. As the Banach space $\mathcal C_0$ has a countable orthonormal basis $\{\tilde{\pi}^{| u|}X^{ u}\}$, by \cite[Theorem 6.10]{M1} and \cite[Equation 1.19]{AS}, we have the following Dwork trace formula
\begin{equation}\label{trace}
L\big(f(X),T\big)^{(-1)^{n+1}}=\det(1-\alpha_aT|_{\mathcal{C}_0})^{\delta^n},\end{equation}
where $g(T)^{\delta}=g(T)/g(qT)$.

\section{Family of exponential sums}

In  what follows, we consider the following Laurent polynomial
$$f(\Lambda,X)=\sum a_{r,u}\Lambda^{r}X^{u} \in \mathbb{F}_q[\Lambda_1^{\pm},\cdots,\Lambda_s^{\pm}, X_1^{\pm},\cdots,X_n^{\pm}],$$
here we use multiple index:
$$\Lambda^r=\Lambda_1^{r_1}\cdots\Lambda_s^{r_s}, X^{u}=X_1^{u_1}\cdots X_n^{u_n}.$$
Let $\text{supp }f:=\{(r,u)\in \mathbb{Z}^{s+n}:a_{r,u}\neq 0\}$.
Let $\delta_1$ be the cone in $\mathbb{R}^s$ generated by
$\{r\in \mathbb Z^s: {\rm there\ exists} \  u\in \mathbb Z^n \text{ such that }(r,u)\in \text{supp }f\},$
$\delta_2$ be the cone in $\mathbb{R}^n$ generated by
$\{u\in \mathbb Z^n: {\rm there\ exists} \  r\in \mathbb Z^s \text{ such that }(r,u)\in \text{supp }f\}.$ Let $M_1=\delta_1\cap \mathbb{Z}^s$ and $M_2=\delta_2\cap \mathbb{Z}^n$. Let
$$\omega_1=\max\{\vert r\vert \in \mathbb Z^s: {\rm there\ exists} \   u\in \mathbb Z^n \text{ such that }(r,u)\in \text{supp }f\}$$
and
$$\omega_2=\max\{\vert u\vert \in \mathbb Z^n: {\rm there\ exists} \  r\in \mathbb Z^s \text{ such that }(r,u)\in \text{supp }f\}.$$

Now fix $\tilde{\pi}\in \overline{\mathbb Q}_p$ such that $\text{ord}_p\tilde{\pi}=\frac{p-1}{p^2(\omega_1+\omega_2)}$.
For integer $m\geq 0$, set
$$\mathcal{O}_{0,p^m}:=\{\sum_{r\in M_1} c(r)\tilde{\pi}^{\frac{\vert r\vert}{p^m}}\Lambda^{r}: c(r)\in R,c(r)\rightarrow  0\ {\rm as}\  \vert r\vert\rightarrow\infty \}.$$
For $\xi=\sum_{r\in M_1} c(r)\tilde{\pi}^{\frac{\vert r\vert}{p^m}}\Lambda^{r}\in \mathcal{O}_{0,p^m}$, its norm is given by
$$\vert \xi\vert= \sup_{r\in M_1}|c({r})|.$$

Similarly, we define the $\mathcal{O}_{0,p^m}$-algebra

$$\mathcal{C}_0(\mathcal{O}_{0,p^m}):=\{\sum_{u\in M_2} \zeta({u})\tilde{\pi}^{|{u}|}X^{{{u}}}: \zeta({u})\in \mathcal{O}_{0,p^m},\zeta({u})\rightarrow  0\ {\rm as}\  |{u}|\rightarrow\infty \}$$
with norm

$$|\sum_{u\in M_2} \zeta({u})\tilde{\pi}^{|{u}|}X^{{{u}}}|:=\sup_{{u}\in M_2}|\zeta({u})|.$$
 For $m=0$, we simply write $\mathcal{O}_{0,1}$ by $\mathcal{O}_{0}$, and $\mathcal{C}_0(\mathcal{O}_{0,1})$ by $\mathcal{C}_0(\mathcal{O}_{0})$.

For $s_1,s_2\in \mathbb{R}$ such that $s_1>1$ and $s_2>1$, we define the following Banach space:

$$K(s_1,s_2)=\big\{\sum_{r\in M_1, u\in M_2}B_{r,u}\Lambda^r X^u: B_{r,u}\in K \ \text{and}\ \{\vert B_{r,u}\vert s_1^{\vert r\vert}s_2^{\vert u\vert}\}_{r,u} \text{ is bounded}\big\}.$$
Then we have that $$\mathcal{C}_0(\mathcal{O}_{0})\subset K(p^{\frac{p-1}{p^2(\omega_1+\omega_2)}},p^{\frac{p-1}{p^2(\omega_1+\omega_2)}}).$$

\subsection{Frobenius}
Let $\hat{f}(\Lambda,X)\in K[\Lambda^\pm,X^\pm]$ be the Teichm\"{u}ller lifting of $f(\Lambda,X)$ and let
$$H_m(\Lambda,X):=\exp\pi \big(\hat{f}(\Lambda, X)-\hat{f}(\Lambda^{p^m}, X^{p^m})\big).$$
As in Lemma \ref{lem01}, one verifies that
\begin{equation}\label{eq1029}
H_m(\Lambda, X)\in K(p^{\frac{p-1}{p^{m+1}(\omega_1+\omega_2)}}, p^{\frac{p-1}{p^{m+1}(\omega_1+\omega_2)}}).
\end{equation}

Let $\psi$ be the map acting on power series by
$$\psi(\sum B_{r,u}\Lambda^r X^u)=\sum B_{r,pu}\Lambda^r X^u.$$
One checks that
$$\psi^m\big(K(s_1,s_2)\big)\subset K(s_1,s_2^{p^m}).$$
Define the Frobenius map
$$\alpha_{m,\Lambda}:=\psi^m\circ H_m(\Lambda, X).$$
\begin{lemma}\label{lem1}
For $m\ge 1$, the map $\alpha_{m,\Lambda}$ maps $\mathcal{C}_0(\mathcal{O}_0)$ into $\mathcal{C}_0(\mathcal{O}_{0,p^m})$, and for $u\in M_2$, we have that
$$|\alpha_{m,\Lambda}(\tilde{\pi}^{|u|}X^{u}) |\le  |\tilde{\pi}|^{(1-\frac{1}{p^m})|u|}.$$
\end{lemma}

\begin{proof}
We have that $\alpha_{m,\Lambda}$ is a composition of the following maps:

$$\mathcal{C}_0(\mathcal{O}_{0})\hookrightarrow K(p^{\frac{p-1}{p^2(\omega_1+\omega_2)}},p^{\frac{p-1}{p^2(\omega_1+\omega_2)}})\stackrel{H_m(\Lambda, X)}\longrightarrow  K(p^{\frac{p-1}{p^{m+1}(\omega_1+\omega_2)}}, p^{\frac{p-1}{p^{m+1}(\omega_1+\omega_2)}})$$

$$\stackrel{\psi^m}\longrightarrow K(p^{\frac{p-1}{p^{m+1}(\omega_1+\omega_2)}}, p^{\frac{p-1}{p(\omega_1+\omega_2)}})\hookrightarrow \mathcal{C}_0(\mathcal{O}_{0,p^m}).$$
Hence $\alpha_{m,\Lambda}$ maps $\mathcal{C}_0(\mathcal{O}_0)$ into $\mathcal{C}_0(\mathcal{O}_{0,p^m})$.

By (\ref{eq1029}), we write $H_m(\Lambda,X)=\sum_{(s,v)\in M_1\times M_2} B_{s,v}\Lambda^s X^v$ with
\begin{equation}\label{eq09091}
\vert B_{s,v} \vert p^{\frac{(p-1)(|s|+|v|)}{p^{m+1}(\omega_1+\omega_2)}}\le 1.
\end{equation}
We have that
\begin{eqnarray}\label{09101}
\alpha_{m,\Lambda}(\tilde{\pi}^{|u|}X^{u})&=&\psi^m\big(\sum_{(s,v)\in M_1\times M_2}B_{s,v}\tilde{\pi}^{|u|}\Lambda^s X^{v+u}\big)=\sum_{(s,v)\in M_1\times M_2}B_{s, p^m v-u}\tilde{\pi}^{|u|}\Lambda^s X^{v}\nonumber\\
&=&\sum_{(s,v)\in M_1\times M_2}B_{s, p^m v-u}\tilde{\pi}^{|u|-|v|-\frac{|s|}{p^m}}(\tilde{\pi}^{\frac{|s|}{p^m}}\Lambda^s) (\tilde{\pi}^{|v|}X^{v}).
\end{eqnarray}
By (\ref{eq09091}), we have that
\begin{eqnarray*}
\big|B_{s, p^m v-u}\tilde{\pi}^{|u|-|v|-\frac{|s|}{p^m}}\big|&\le& \big|\tilde{\pi}\big|^{\frac{|s|+|p^m v-u|}{p^{m-1}}+|u|-\frac{|s|}{p^m}-|v|}\\
&=&\big|\tilde{\pi}\big|^{\frac{(p-1)|s|+p|p^mv-u|+p^m |u|-p^m |v|}{p^m}}\\
&\le&\big|\tilde{\pi}\big|^{\frac{(p-1)|s|+|p^mv-u|+p^m |u|-p^m |v|}{p^m}}\\
&\le&\big|\tilde{\pi}\big|^{\frac{(p-1)|s|+(p^m-1) |u|}{p^m}}\\
&\le&\big|\tilde{\pi}\big|^{\frac{(p^m-1) |u|}{p^m}}.
\end{eqnarray*}
It follows that
$$|\alpha_{m,\Lambda}(\tilde{\pi}^{|u|}X^{u}) |\le  |\tilde{\pi}|^{(1-\frac{1}{p^m})|u|}.$$
This finishes the proof of Lemma \ref{lem1}.
\end{proof}

Another estimate for the $p$-adic value of $\alpha_{m,\Lambda}(\tilde{\pi}^{|u|}X^{u})$ is given by the following.

\begin{lemma}\label{lem2}
Write $$\alpha_{m,\Lambda}(\tilde{\pi}^{|u|}X^{u})=\sum_{(s,v)\in M_1\times M_2} D^{(u)}_{s,v}\tilde{\pi}^{\frac{|s|}{p^m}}\Lambda^s\tilde{\pi}^{|v|}X^{v}.$$
Then we have that $|D^{(u)}_{s,v}|\le |\tilde{\pi}|^{(p-1)(\frac{|s|}{p^m}+|v|)}$. In particular, $|\alpha_{m,\Lambda}(1)-1|<1$.
\end{lemma}

\begin{proof}
It follows from (\ref{09101}) that $$D^{(u)}_{s,v}=B_{s, p^m v-u}\tilde{\pi}^{|u|-|v|-\frac{|s|}{p^m}}.$$
Then by (\ref{eq09091}), we have that
\begin{eqnarray*}
\big|D^{(u)}_{s,v}\big|&\le& \big|\tilde{\pi}\big|^{\frac{|s|+|p^m v-u|}{p^{m-1}}+|u|-\frac{|s|}{p^m}-|v|}
=\big|\tilde{\pi}\big|^{\frac{(p-1)|s|+p|p^mv-u|+p^m |u|-p^m |v|}{p^m}}\\
&\le&\big|\tilde{\pi}\big|^{\frac{(p-1)|s|+p|p^mv-u|+p |u|-p^m |v|}{p^m}}\\
&\le& |\tilde{\pi}|^{\frac{p-1}{p^m}|s|+(p-1)|v|}.
\end{eqnarray*}
Note that $D_{0,0}^{(0)}=B_{0,0}$ and $B_{0,0}\equiv 1 \mod \pi$.
We have that $D_{0,0}^{(0)}\equiv 1 \mod \tilde{\pi}$. It follows that $|\alpha_{m,\Lambda}(1)-1|<1$. This finishes the proof of Lemma \ref{lem2}.
\end{proof}

\subsection{Fibres}\label{fibre}

Fix $\lambda\in \bar{\mathbb{F}}_q^s$. Let
$\alpha_{ad(\lambda), \lambda}:=\psi^{a d(\lambda)}\circ H_{a d(\lambda)}(\hat\lambda,X)$, where $\hat\lambda\in  \overline{\mathbb Q}_p$ is the Teichm\"uller lifting of $\lambda$.
By Lemma \ref{lem2}, we conclude that $\alpha_{ad(\lambda), \lambda}$ is a completely continuous map from $\mathcal{C}_0(\lambda)$ to
$\mathcal{C}_0(\lambda)$, where
$\mathcal{C}_0(\lambda)=\mathcal{C}_0(\mathcal{O}_0) \otimes K(\lambda) $.
The Dwork trace formula (\ref{trace}) states that
$$L(f(\lambda,X),T)^{(-1)^{n+1}}= \det(1-\alpha_{ad(\lambda), \lambda}T|_{\mathcal{C}_0(\lambda)})^{\delta_{q^{d(\lambda)}}^n},$$
where $g(T)^{\delta_{q^{d(\lambda)}}}:=g(T)/g(q^{d(\lambda)}T)$.
It is well-known that this $L$-function is rational. By \cite[Corollary 3.12]{AS}, the $L$-function $L(\lambda,T)^{(-1)^{n+1}}$ has a unique reciprocal root, which is a $p$-adic $1$-unit, denoted by  $\pi_0(\lambda )$.
For a $p$-adic integer $\kappa$, the unit root $L$-function associated to the family of $f(\Lambda, X)$ is defined by
$$L_{unit}(f, \kappa, T):=\prod_{\lambda\in |\mathbb{G}_m^s/\mathbb{F}_q|}\frac{1}{1-\pi_0(\lambda )^{\kappa}T^{d(\lambda)}}.$$

\subsection{Infinite symmetric power}\label{3.5}
For a positive integer $m$, let $\mathcal{S}(\mathcal{O}_{0,p^m}):=\mathcal{O}_{0.p^m}[[\{e_{u}\}_{u\in M_2\setminus\{0\}}]]$ be the formal power series ring over $\mathcal{O}_{0,p^m}$ with variables $\{e_u\}_{u\in M_2\setminus \{0\}}$. We equip this ring with the sup-norm on the coefficients.
For $\mathbf u=(u_1,\cdots, u_t)$, let $e_{\mathbf{u}}:=e_{u_1}\cdots e_{u_t}$ be the monomial of degree $t$. Let $\mathcal{S}(M_2)$ be the set of $\mathbf{u}$ corresponding to $e_{\mathbf{u}}$. Let $$|\mathbf{u}|:=|u_1|+\cdots+|u_t|.$$
Define a subspace $\mathcal{S}_0(\mathcal{O}_{0,p^m})$ of $\mathcal{S}(\mathcal{O}_{0,p^m})$ by
$$\mathcal{S}_0(\mathcal{O}_{0,p^m}):=\{\zeta=\sum_{\mathbf{u}\in  \mathcal{S}(M_2)} \zeta(\mathbf{u})e_{\mathbf{u}}: \zeta(\mathbf{u})\in \mathcal{O}_{0,p^m},  \zeta(\mathbf{u})\rightarrow 0 \ {\rm as}\ |\mathbf{u}| \rightarrow \infty\}.$$
We embed $\mathcal{C}_0(\mathcal{O}_{0,p^m})$ into $\mathcal{S}_0(\mathcal{O}_{0,p^m})$ by $\Upsilon(\tilde{\pi}^{|u|}X^{u})=e_{u}. $
From Lemma \ref{lem2}, we know that  $\alpha_{m,\Lambda}(1)=1+\eta_m$ with $\eta_m\in \mathcal{C}_0(\mathcal{O}_{0,p^m})$ and $|\eta_m|<1$. Hence for any $p$-adic integer $\kappa$, $(\Upsilon\circ \alpha_{m,\Lambda}(1))^{\kappa}\in \mathcal{S}_0(\mathcal{O}_{0,p^m})$ is well-defined.
Now we can extend the map $\alpha_{m,\Lambda}$ to a map $[\alpha_{m,\Lambda}]_{\kappa}$ from $\mathcal{S}_0(\mathcal{O}_0)$ to $\mathcal{S}_0(\mathcal{O}_{0,p^m})$ by
$$ [\alpha_{m,\Lambda}]_{\kappa}(e_{\mathbf{u}}):= \big( \Upsilon\circ\alpha_{m,\Lambda}(1)\big)^{\kappa-t} \big( \Upsilon\circ\alpha_{m,\Lambda}(\tilde{\pi}^{|u_1|}X^{u_1} )\big)\cdots \big( \Upsilon\circ\alpha_{m,\Lambda}(\tilde{\pi}^{|u_t|}X^{u_t} )\big). $$
We have the following estimate for $[\alpha_{m,\Lambda}]_{\kappa}$.

\begin{lemma}\label{lem3}
For $\mathbf{u}=(u_1,\ldots,u_t)$, if we write $[\alpha_{m,\Lambda}]_{\kappa}(e_{\mathbf{u}})=\sum_{(s,\mathbf{v})\in M_1\times \mathcal{S}(M_2)}D_{s,\mathbf{v}}\tilde{\pi}^{\frac{|s|}{p^m}}\Lambda^se_{\mathbf{v}}$, then
$$ {\rm ord}_{\tilde{\pi}} D_{s,\mathbf{v}}\ge \frac{p-1}{p^m}|s|+(p-1)|\mathbf{v}|.$$
\end{lemma}

\begin{proof}
By Lemma \ref{lem2}, we write
$$\Upsilon\big(\alpha_{m,\Lambda}(1)\big)=\sum_{(s,v)}D_{s,v}^{(0)}(\tilde{\pi}^{\frac{|s|}{p^m}}\Lambda^s)e_v$$
and
$$\Upsilon\big(\alpha_{m,\Lambda}(\tilde{\pi}^{|u_j|}X^{u_i})\big)=\sum_{(s,v)}D_{s,v}^{(i)}(\tilde{\pi}^{\frac{|s|}{p^m}}\Lambda^s)e_v,$$
where
$|D^{(j)}_{s,v}|\le |\tilde{\pi}|^{\frac{p-1}{p^m}|s|+(p-1)|v|}$ for $j=0,1,\ldots,t$.
We first consider the case that $\kappa$ is a positive integer such that $t\le \kappa$. We have that
\begin{eqnarray*}
[\alpha_{m,\Lambda}]_{\kappa}(e_{\mathbf{u}})
&=&\big(\sum_{(s,v)}D_{s,v}^{(0)}(\tilde{\pi}^{\frac{|s|}{p^m}}\Lambda^s)e_v\big)^{\kappa-t}\big(\sum_{(s,v)}D_{s,v}^{(1)}(\tilde{\pi}^{\frac{|s|}{p^m}}\Lambda^s)e_v\big)\cdots\big(\sum_{(s,v)}D_{(s,v)}^{(t)}(\tilde{\pi}^{\frac{|s|}{p^m}}\Lambda^s)e_v\big)\\
&=&\sum_{\stackrel{s_1+\ldots+s_{\kappa}=s}{\mathbf{v}=(v_1, \ldots,v_\kappa)}}D^{(j_1)}_{s_1,v_1}D^{(j_2)}_{s_2,v_2}\cdots D^{(j_k)}_{s_\kappa,v_\kappa}\tilde{\pi}^{\frac{|s_1|+\ldots+|s_\kappa|-|s|}{p^m}}\big(\tilde{\pi}^{\frac{|s|}{p^m}}\Lambda^s\big)e_{\mathbf v},
\end{eqnarray*}
where  $j_1,\ldots,j_{\kappa}$  take values in $\{0,1,\ldots,t\}$.
Then
\begin{eqnarray*}
\text{ord}_{\tilde{\pi}}\big(D_{s,\mathbf{v}}\big)&\ge &\min_{\stackrel{s_1+\ldots+s_{\kappa}=s}{\mathbf{v}=(v_1, \ldots,v_\kappa)}}\text{ord}_{\tilde{\pi}}\big(D^{(j_1)}_{s_1,v_1}D^{(j_2)}_{s_2,v_2}\cdots D^{(j_k)}_{s_\kappa,v_\kappa}\tilde{\pi}^{\frac{|s_1|+\ldots+|s_\kappa|-|s|}{p^m}}\big)\\
&\geq&\frac{p-1}{p^m}(|s_1|+\ldots+|s_\kappa|)+(p-1)(|v_1+\ldots+|v_\kappa|)+\frac{|s_1|+\ldots+|s_\kappa|-|s|}{p^m}\\
&\geq&\frac{p-1}{p^m}|s|+(p-1)|\mathbf{v}|.
\end{eqnarray*}
For the case that $\kappa$ is a $p$-adic integer, we take a positive integers sequence $\{k_n\}$ tending to $\infty$ in the archimedean sense and tending to $\kappa$ p-adically. By Lemma \ref{lem6} below, we have
$$\lim_{n\to \infty}[\alpha_{m,\Lambda}]_{k_n}=[\alpha_{m,\Lambda}]_{\kappa},$$
which means Lemma \ref{lem3} is valid for all $p$-adic integer $\kappa$.

\end{proof}

Let $[\alpha_{ad(\lambda),\lambda}]_{\kappa}$ be the specialization of $[\alpha_{ad(\lambda),\Lambda}]_{\kappa}$ at $\lambda$. Write
$$[\alpha_{ad(\lambda),\lambda}]_{\kappa}(e_{\mathbf{u}})=\sum_{\mathbf v}{D}_{\mathbf{v}}e_{\mathbf{v}}$$
by Lemma \ref{lem3}, we have ${\rm ord}_{\tilde{\pi}}D_{\mathbf{v}}\ge (p-1)|\mathbf{v}|$, which means
$[\alpha_{ad(\lambda),\lambda}]_{\kappa}$ is a completely continuous operator of $\mathcal{S}_0(\mathcal{O}_0)\otimes K(\lambda)$.
Hence we can define the $L$-function associated to $[\alpha_{a,\Lambda}]_{\kappa}$ by the following
$$L([\alpha_a]_{\kappa}, T):=\prod_{\lambda\in |\mathbb{G}_m^s/\mathbb{F}_q|}\frac{1}{\det (1-T[\alpha_{ad(\lambda),\lambda}]_{\kappa}|_{\mathcal{S}_0(\mathcal{O}_0)\otimes K(\lambda)})}.$$

For positive integer $i$, we can similarly  define an operator
$$[\alpha_{a,\Lambda}]_{\kappa-i}\otimes \wedge^i \alpha_{a,\Lambda}:\mathcal{S}_0(\mathcal{O}_0)\otimes\wedge^i \mathcal{C}_0(\mathcal{O}_{0})\to \mathcal{S}_0(\mathcal{O}_{0,q})\otimes\wedge^i\mathcal{C}_0(\mathcal{O}_{0,q})$$
and its associated $L$-function $L([\alpha_a]_{\kappa-i}\otimes \wedge^i \alpha_a,T)$.
By Lemma 2.1 and Corollary 2.4 in \cite{H14}, we have that
$$L_{unit}(f, \kappa, T)=L([\alpha_a]_{\kappa},T) \prod_{i\ge 2} L([\alpha_a]_{\kappa-i}\otimes \wedge^i \alpha_a,T)^{(-1)^{i-1}(i-1)}.$$
Since $\alpha_{a,\Lambda}(1)\equiv 1 \mod \tilde{\pi}$ and $\alpha_{a,\Lambda}(\tilde{\pi}^uX^u)\equiv 0 \mod \tilde{\pi}$ for $u\in M_2\setminus\{0\}$, one can see that the unit root of $L_{unit}(f, \kappa, T)$ only comes from $L([\alpha_a]_{\kappa},T)$.

We now view $\mathcal{S}_0(\mathcal{O}_{0})$ as a $p$-adic Banach space over $K$ with orthonormal basis $$\{ \tilde{\pi}^{|r|}\Lambda^{r}e_{\mathbf{u}}:r\in M_1, \mathbf{u}\in \mathcal{S}(M_2)\}.$$
Consider the map $\psi_{\Lambda}$ defined by
$$\psi_{\Lambda}(\sum D_{r,\mathbf v}\Lambda^{r}e_{\mathbf{u}})= \sum D_{pr,\mathbf v}\Lambda^{r}e_{\mathbf{u}}.$$
One can check that
$$\psi_{\Lambda}^m\big(\mathcal{S}_0(\mathcal{O}_{0,p^m})\big)\subset \mathcal{S}_0(\mathcal{O}_0).$$
Now we define the Frobenius map $[\beta_a]_{\kappa}$ acting on $\mathcal{S}_0(\mathcal{O}_0)$ by $[\beta_a]_{\kappa}:=\psi_{\Lambda}^a\circ [\alpha_{a,\Lambda}]_{\kappa}$.

 \begin{lemma}\label{lem3.7}
 The map $[\beta_a]_{\kappa}$ is completely continuous.
 \end{lemma}
\begin{proof}
It is enough to consider $[\beta_a]_{\kappa}$ acting on the orthonormal basis $\{\tilde{\pi}^{|r|}\Lambda^{r}e_{\mathbf{u}}\}_{r,\mathbf{u}}$.
\begin{align*}
[\beta_a]_{\kappa}( \tilde{\pi}^{|r|}\Lambda^{r}e_{\mathbf{u}})&=\psi_{\Lambda}^a(\tilde{\pi}^{|r|}\Lambda^r [\alpha_{a,\Lambda}]_{\kappa}(e_{\mathbf{u}}))
=\psi_{\Lambda}^a\big(\sum_{(s,\mathbf{v})}D_{s,\mathbf{v}}\tilde{\pi}^{|r|+\frac{|s|}{p^a}}\Lambda^{s+r}e_{\mathbf{v}}\big)\\
&=\sum_{(s,\mathbf{v})}D_{p^as-r,\mathbf{v}}\tilde{\pi}^{|r|+\frac{|p^as-r|}{p^a}-|s|}(\tilde{\pi}^{|s|}\Lambda^se_{\mathbf{v}}).
\end{align*}
By Lemma \ref{lem3} we have that
\begin{eqnarray*}
{\rm ord}_{\tilde{\pi}}D_{p^as-r,\mathbf{v}}\tilde{\pi}^{|r|+\frac{|p^as-r|}{p^a}-|s|}&\ge& \frac{p-1}{p^a}|p^as-r|+(p-1)|\mathbf v|+|r|+\frac{|p^as-r|}{p^a}-|s|\\
&\ge&\frac{1}{p^{a-1}}|p^as-r|+\frac{1}{p^{a-1}}|r|-|s|+(p-1)|\mathbf v|\\
&\ge&(p-1)(|s|+|\mathbf{v}|).
\end{eqnarray*}
The estimates of coefficients are independent on $r$ and $\mathbf u$, which means $[\beta_a]_{\kappa}$ is completely continuous.
\end{proof}

Let $\mathcal{B}=\{e_{\mathbf{u}} |\mathbf{u} \in \mathcal{S}(M_2)\}$. Let $B^{[\kappa]}(\Lambda)$ be the matrix of $[\alpha_{a,\Lambda}]_{\kappa}$ with respect to $\mathcal{B}$. Then we can write
$$B^{[\kappa]}(\Lambda)=\sum_{r\in M_1}b_{r}^{[\kappa]} \Lambda^{r}.$$
 Let $F_{B^{[\kappa]}}:=(b_{qr-s}^{[\kappa]})_{r, s \in M_1}$ and $b_{qr-s}^{[\kappa]}:=0$ when $qr-s\notin M_1$. Then $F_{B^{[\kappa]}}$ is the matrix of $[\beta_a]_{\kappa}$ with respect to $\{\Lambda^se_{\mathbf{v}}\}_{s\in M_1,\mathbf{v}\in \mathcal{S}(M_2)}$. The Dwork's trace formula \cite[Lemma 4.1]{Wan96} states that
 $$(q^{m}-1)^s {\rm Tr}([\beta_a]_{\kappa}^m )=(q^{m}-1)^s {\rm Tr}(F_{B^{[\kappa]}}^m)=\sum_{\lambda\in (\mathbb{F}_{q^{m}}^{*})^s,\hat{\lambda}={\rm Teich} (\lambda)}{\rm Tr}([\alpha_{ad(\lambda),\lambda}]_{\kappa}^m|S_0(\hat{\lambda})).$$
Then by the same argument as in \cite[section 2]{H14}, we have that
$$L([\alpha_a]_{\kappa},T)^{(-1)^{s+1}}=\det(1-[\beta_a]_{\kappa}T)^{\delta_{q}^s}.$$
It follows that the unit root of $L_{unit}(f, \kappa, T)^{(-1)^{s+1}}$ comes from the factor
$\det(1-[\beta_a]_{\kappa}T)$ and the matrix of $[\beta_a]_{\kappa}$  shows that $\det(1-[\beta_a]_{\kappa}T)$ has a unique unit root.
Thus we prove the following result.
\begin{theorem}\label{thm1}
The unit root L-function  $L_{unit}(f, \kappa, T)^{(-1)^{s+1}}$ has one unique unit root, and this unit root comes from $\det(1-[\beta_a]_{\kappa}T)$.
\end{theorem}

\section{Dual theory}
For integer $m\ge 1$, we define the $\mathcal{O}_{0,p^m}$-module
$$\mathcal{C}_0^*(\mathcal{O}_{0,p^m}):=\{\sum_{u\in M_2} \zeta(u)\tilde{\pi}^{-|u|}X^{-u}: \zeta(u)\in \mathcal{O}_{0,p^m}\}$$
equipped with the sup-norm on the coefficients.
Let $\Phi_{X}$ be the Frobenius map in variable $X$ defined by $\Phi_{X}(X^{u})=X^{pu}$, and let ${\rm Pr}_2$ be the projection map defined by
$${\rm Pr}_2(\sum_{u\in \mathbb{Z}^n}A(u)X^{-u} )=\sum_{u\in M_2} A(u)X^{-u} .$$
Define the Frobenius map $\alpha_{m,\Lambda}^*$ by
$$\alpha_{m,\Lambda}^*={\rm Pr}_2\circ H_m(\Lambda,X)\circ \Phi_{X}^m.$$

\begin{lemma} \label{lem4}
For integer $m\ge 1$, $\alpha_{m,\Lambda}^*$ maps $\mathcal{C}_0^*(\mathcal{O}_{0,p^m})$ to $\mathcal{C}_0^*(\mathcal{O}_{0,p^m})$ linearly over $\mathcal{O}_{0,p^m}$. Furthermore, $$|\alpha_{m,\Lambda}^*(\tilde{\pi}^{-|u|}X^{-u})|\le |\tilde{\pi}|^{(p-1)|u|}$$ and
$|\alpha_{m,\Lambda}^*(1)-1|<1$.
\end{lemma}

\begin{proof}
First we consider $\alpha_{m,\Lambda}^*( \tilde{\pi}^{-|u|}X^{-u})$ with $u\in M_2$.  Recall that
$$H_m(\Lambda,X)=\sum_{(s,v)\in M_1\times M_2} B_{s,v}\Lambda^sX^v$$
with
$|B_{s,v}|\le |\tilde{\pi}|^{\frac{|s|+|v|}{p^{m-1}}}$. Then we have
\begin{align*}
\alpha_{m,\Lambda}^*( \tilde{\pi}^{-|u|}X^{-u})&={\rm Pr}_2(H_m(\Lambda,X) \cdot \tilde{\pi}^{-|u|}X^{-p^mu})\\
&={\rm Pr}_2(\sum_{(s,v)\in M_1\times M_2} B_{s,v}\tilde{\pi}^{-|u|}\Lambda^sX^{v-p^mu} )\\
&=\sum_{(s,v)\in M_1\times M_2}B_{s,p^mu-v}\tilde{\pi}^{-|u|+|v|-\frac{|s|}{p^m}}(\tilde{\pi}^{\frac{|s|}{p^m}}\Lambda^s)(\tilde{\pi}^{-|v|}X^{-v}).
\end{align*}
It follows that $$|B_{s,p^mu-v}\tilde{\pi}^{-|u|+|v|-\frac{|s|}{p^m}}|\le |\tilde{\pi}|^{\frac{|s|+|p^mu-v|}{p^{m-1}}-|u|+|v|-\frac{|s|}{p^m}}=
|\tilde{\pi}|^{\frac{p|p^mu-v|-p^m|u|+p^m|v|+(p-1)|s|}{p^{m}}}.$$
Note that $|p^mu-v|+|v|\ge p^m|u|$. We have the following two estimates:
\begin{equation}\label{09130}
|B_{s,p^mu-v}\tilde{\pi}^{-|u|+|v|-\frac{|s|}{p^m}}|\le |\tilde{\pi}|^{(p-1)|u|}
\end{equation}
and
\begin{equation}\label{09131}
|B_{s,p^mu-v}\tilde{\pi}^{-|u|+|v|-\frac{|s|}{p^m}}|\le |\tilde{\pi}|^{\frac{(p-1)|s|+(p^m-1)|v|}{p^m}}.
\end{equation}
From (\ref{09130}) we get $$|\alpha_{m,\Lambda}^\ast(\tilde{\pi}^{-|u|}X^{-u})|\le |\tilde{\pi}|^{(p-1)|u|}.$$
From (\ref{09131}) we get
$$\alpha_{m,\Lambda}^*(\tilde{\pi}^{-|u|}X^{-u})\in \mathcal{C}_0^*(\mathcal{O}_{0,p^m}).$$

Hence we can write $\alpha_{m,\Lambda}^*(\tilde{\pi}^{-|u|}X^{-u})=\tilde{\pi}^{(p-1)|u|}\eta_u(\Lambda,X)$ with $\eta_u(\Lambda,X)\in \mathcal{C}_0^*(\mathcal{O}_{0,p^m})$. For
$\zeta^\ast=\sum_u\zeta(u)\tilde{\pi}^{-|u|}X^{-u}\in\mathcal{C}_0^*(\mathcal{O}_{0,p^m})$, we have that
$$\alpha_{m,\Lambda}^*(\zeta^\ast)=\sum_u\tilde{\pi}^{(p-1)|u|}\zeta(u)\eta_u(\Lambda,X)	\in \mathcal{C}_0^*(\mathcal{O}_{0,p^m}).$$
It follows from $B_{0,0}\equiv 1\mod \tilde{\pi}$ and (\ref{09131}) that $|\alpha_{m,\Lambda}^*(1)-1|<1$.
Hence Lemma \ref{lem4} holds.

\end{proof}

Let $m_1,m_2$ be two positive integers and $m'=\max\{m_1,m_2\}$. Define the pairing $$\langle\cdot, \cdot\rangle:\mathcal{C}_0(\mathcal{O}_{0,p^{m_1}})\times \mathcal{C}_0^*(\mathcal{O}_{0,p^{m_2}})\rightarrow \mathcal{O}_{0,p^{m'}}$$ by
$$\langle\zeta,\zeta^*\rangle =\ {\rm the\ constant\ term\ of}\  \zeta\cdot \zeta^* \ {\rm with\ respect\ to}\ X,$$
where $\zeta=\sum\zeta_1(v)\tilde{\pi}^{|v|}X^v\in \mathcal{C}_0(\mathcal{O}_{0,p^{m_1}})$ and
$\zeta^*=\sum\zeta_2(v)\tilde{\pi}^{-|v|}X^{-v}\in \mathcal{C}_0^*(\mathcal{O}_{0,p^{m_2}})$. Since $\{\zeta_1(v)\}\subset \mathcal{O}_{0,p^{m_1}}$ with $\zeta_1(v)\rightarrow 0$ as $v\rightarrow \infty$ and $\{\zeta_2(v)\}\subset \mathcal{O}_{0,p^{m_2}}$ which is bounded, the product is well-defined.

For $\zeta\in \mathcal{C}(\mathcal{O}_0)$ and $\zeta^*\in \mathcal{C}^*(\mathcal{O}_{0,p^{m}})$, we can check that
$$\langle(\psi_{X}^m\circ H_m)\zeta,\zeta^*\rangle =\langle H_m\zeta,\Phi_X^m(\zeta^*) \rangle =\langle\zeta, {\rm Pr}_2(H_m\circ \Phi_X^m(\zeta^*))\rangle .$$

\subsection{Dual Infinity Symmetric power}
With notations in section \ref{3.5}, define
$$\mathcal{S}_0^*(\mathcal{O}_0):=\{\sum_{\mathbf{u}\in \mathcal{S}(M_2)}\zeta^*(\mathbf{u})e_{\mathbf{u}}^\ast: \zeta^*(\mathbf{u})\in \mathcal{O}_0\}. $$
It is the formal power series ring over $\mathcal{O}_0$ in the variables $\{e_{u}^*\}_{u\in M_2\setminus \{0\}}$ equipped with the sup-norm on the coefficients in $\mathcal O_0$. Similarly, for positive integer $m$, we define
$$\mathcal{S}_0^*(\mathcal{O}_{0,p^m}):=\{\sum_{\mathbf{u}\in \mathcal{S}(M_2)}\zeta^*(\mathbf{u})e_{\mathbf{u}}^\ast: \zeta^*(\mathbf{u})\in \mathcal{O}_{0,p^m}\}. $$
We embed $\mathcal{C}_0^*(\mathcal{O}_0)$ into $\mathcal{S}_0^*(\mathcal{O}_0)$ by $\Upsilon(\tilde{\pi}^{-|u|}X^{-u})=e_{u}^*$.
From Lemma \ref{lem4}, we know that
$\alpha_{m,\Lambda}^*(1)=1+\eta_m^*$ with $\eta_m^*\in \mathcal{C}_0^*(\mathcal{O}_{0,p^m})$ and $|\eta_m^*|<1$. Hence for $\kappa \in \mathbb{Z}_p$, $(\Upsilon\circ \alpha_{m,\Lambda}^*(1))^{\kappa}\in \mathcal{S}_0^*(\mathcal{O}_{0,p^m})$ is well-defined. Now we extend the map $\alpha_{m,\Lambda}^\ast$ to a map from $S_0^*(\mathcal{O}_{0,p^m})$ to $S_0^*(\mathcal{O}_{0,p^m})$ by
\begin{equation}\label{eq0923}
[\alpha_{m,\Lambda}^*]_{\kappa}(e_{u_1}^*\cdots e_{u_r}^*):=\big(\Upsilon\circ\alpha_{m,\Lambda}^*(1)\big) ^{\kappa-r} \big(\Upsilon\circ\alpha_{m,\Lambda}^*(\tilde{\pi}^{-|u_1|}X^{-u_1})\big)\cdots \big(\Upsilon\circ\alpha_{m,\Lambda}^*(\tilde{\pi}^{-|u_r|}X^{-u_r})\big).
\end{equation}
From Lemma \ref{lem4}, we conclude that
\begin{equation}\label{eq09181}
\big|[\alpha_{m,\Lambda}^*]_{\kappa}(e_{u_1}^*\cdots e_{u_r}^*)\big|\le \tilde{\pi}^{(p-1)(|u_1|+\cdots+|u_r|)}.
\end{equation}

For $m\ge 1$, we define a new $R$-module
$$\mathcal{O}_{0,p^m}^*:=\{\sum_{r\in M_1}a^*(r)\tilde{\pi}^{-\frac{|r|}{p^m}}\Lambda^{-r}: a^*(r)\in R\}$$
with the sup-norm defined by
$$|\sum_r a^*(r)\tilde{\pi}^{-\frac{|r|}{p^m}}\Lambda^{-r}|=\sup_r |a^*(r)|.$$
Let
$$\mathcal{S}_0^*(\mathcal{O}_{0,p^m}^*):=\Big\{\sum_{r\in M_1,\mathbf{u}\in \mathcal{S}(M_2)}\zeta^*(r,\mathbf{u})\tilde{\pi}^{-\frac{|r|}{p^m}}\Lambda^{-r}e_{\mathbf{u}}^\ast: \zeta^*(r,\mathbf{u})\in R\Big\}. $$

Let $\Phi_{\Lambda}$ be the Frobenius map in variable $\Lambda$ defined by $\Phi_{\Lambda}(\Lambda):=\Lambda^p$, and let ${\rm Pr}_1$ be the projection map on $M_1$ defined by $${\rm Pr}_1(\sum_{r\in \mathbb{Z}^s}b(r)\Lambda^{-r}):=\sum_{r\in M_1}b(r)\Lambda^{-r}$$
Define $$[\beta_{a}^*]_{\kappa}:={\rm Pr}_1\circ [\alpha_{a,\Lambda}^*]_{\kappa}\circ \Phi_{\Lambda}^a.$$

\begin{lemma}\label{lem5}
Then $[\beta_{a}^*]_{\kappa}$ is an $R$-linear map from $\mathcal{S}_0^*(\mathcal{O}_{0}^*)$ to $\mathcal{S}_0^*(\mathcal{O}_{0}^*)$.
\end{lemma}
\begin{proof}
Note that $[\alpha_{a,\Lambda}^*]_{\kappa}$ is an $\mathcal{O}_{0,q}$-linear endomorphism of $\mathcal{S}_0^*(\mathcal{O}_{0,q})$, and from (\ref{eq09181}) we can write

$$[\alpha_{a,\Lambda}^*]_{\kappa}(e_{\mathbf{u}}^*)=\sum_{s\in M_1, \mathbf{v}\in \mathcal{S}(M_2)}B_{\mathbf{u}}(s,\mathbf{v})\tilde{\pi}^{\frac{|s|}{q}}\Lambda^{s}e_{\mathbf{v}}^*$$
with
$B_{\mathbf{u}}(s,\mathbf{v})\in R$
and $B_{\mathbf{u}}(s,\mathbf{v})\rightarrow 0$ as $|s|+|\mathbf{u}|\rightarrow \infty$.

For $\zeta^\ast=\sum_{r\in M_1,\mathbf{u}\in \mathcal{S}(M_2)}\zeta^*(r,\mathbf{u})\tilde{\pi}^{-|r|}\Lambda^{-r}e_{\mathbf{u}}^\ast\in \mathcal{S}_0^*(\mathcal{O}_{0}^*)$, we have
\begin{align*}
[\beta_{a}^*]_{\kappa}(\zeta^\ast)&={\rm Pr}_1\bigg(\sum_{r\in M_1,\mathbf{u}\in \mathcal{S}(M_2)}\zeta^*(r,\mathbf{u})\tilde{\pi}^{-|r|}\Lambda^{-qr}[\alpha_{a,\Lambda}^*]_{\kappa}(e_{\mathbf{u}}^*)\bigg)\\
&={\rm Pr}_1\bigg(\sum_{r\in M_1,\mathbf{u}\in \mathcal{S}(M_2)}\zeta^*(r,\mathbf{u})\tilde{\pi}^{-|r|}\Lambda^{-qr}\sum_{s\in M_1, \mathbf{v}\in \mathcal{S}(M_2)}B_{\mathbf{u}}(s,\mathbf{v})\tilde{\pi}^{\frac{|s|}{q}}\Lambda^{s}e_{\mathbf{v}}^*\bigg)\\
&=\sum_{t\in M_1,\mathbf{v}\in \mathcal{S}(M_2)}C(t,\mathbf v)\tilde{\pi}^{-|t|}\Lambda^{-t}e_{\mathbf{v}}^*,
\end{align*}
where $$C(t,\mathbf v)=\sum_{\mathbf{u}\in \mathcal{S}(M_2)}\sum_{t=qr-s}\zeta^*(r,\mathbf{u})B_{\mathbf{u}}(s,\mathbf{v})\tilde{\pi}^{-|r|+\frac{|s|}{q}+|t|}.$$
Since
$$-|r|+\frac{|s|}{q}+|t|\ge (1-\frac{1}{q})|t|\ge0,$$
we have that $\zeta^*(r,\mathbf{u})$ and $\tilde{\pi}^{-|r|+\frac{|s|}{q}+|t|}$ are bounded by $1$. Combining with $B_{\mathbf{u}}(s,\mathbf{v})\to 0$ as $|s|+|\mathbf{u}|\rightarrow \infty$, we have $C(t,\mathbf v)\in R$ is well-defined. Hence $[\beta_{a}^*]_{\kappa}(\zeta^\ast)\in \mathcal{S}_0^*(\mathcal{O}_{0}^*)$.

\end{proof}

\subsection{Estimation using finite symmetric powers}

In this subsection, we demonstrate that both $[\beta_a]_{\kappa}$ and $[\beta_a^\ast]_{\kappa}$ are limits of operators coming from finite symmetric powers. Eventually, we show that $[\beta_a]_{\kappa}$ and $[\beta_a^\ast]_{\kappa}$ yield identical Fredholm determinants.
Let $k$ be a positive integer. We define the degree of $\zeta\in \mathcal{S}_0(\mathcal{O}_0)$ (resp. $\zeta^\ast\in \mathcal{S}_0^{*}(\mathcal{O}_0)$) to be the supremum of the length of monomials $e_{\mathbf u}$ (resp. $e^\ast_{\mathbf u}$) appearing with nonzero coefficients.
Define
\begin{eqnarray*}
\mathcal{S}_0^{(k)}(\mathcal{O}_0)&:=&\{\zeta\in \mathcal{S}_0(\mathcal{O}_0): {\rm degree \ of\ } \zeta \le k\},\\
\mathcal{S}_0^{*(k)}(\mathcal{O}_0)&:=&\{\zeta^*\in \mathcal{S}_0^{*}(\mathcal{O}_0): {\rm degree \ of\ } \zeta^* \le k\}.
\end{eqnarray*}

Let $[\alpha_{a,\Lambda}]_{(k)}$ be a semi-linear map from $\mathcal{S}_0^{(k)}(\mathcal{O}_0)$ to $\mathcal{S}_0^{(k)}(\mathcal{O}_{0,q})$ defined by
$$[\alpha_{a,\Lambda}]_{(k)}(e_{u_1}\cdots e_{u_r}):=\big(\Upsilon\circ\alpha_{a,\Lambda}(1)\big)^{k-r}\big(\Upsilon\circ \alpha_{a,\Lambda}(\tilde{\pi}^{|u_1|}X^{u_1})\big)\cdots \big(\Upsilon\circ\alpha_{a,\Lambda}(\tilde{\pi}^{|u_r|}X^{u_r})\big) .$$
In order to take limit of the sequence $[\alpha_{a,\Lambda}]_{(k)}$ as $k$ goes infinity, we extend $[\alpha_{a,\Lambda}]_{(k)}$ to an operator on $\mathcal{S}_0(\mathcal{O}_0)$ by requiring
\begin{align*}
[\alpha_{a,\Lambda}]_{(k)}(e_{u_1}\cdots e_{u_r}):=&\left\{
  \begin{array}{ll}
[\alpha_{a,\Lambda}]_k(e_{u_1}\cdots e_{u_r}), &  \hbox{if  $r\le k$};\\
 0, & \hbox{otherwise}.
\end{array}
\right.
\end{align*}
Similarly, we define $[\alpha_{a,\Lambda}^*]_{(k)}$ acting on $\mathcal{S}_0^{*(k)}(\mathcal{O}_{0,q})$ by
$$[\alpha_{a,\Lambda}^*]_{(k)}(e_{u_1}^*\cdots e_{u_r}^*):=\big(\Upsilon\circ\alpha_{a,\Lambda}^\ast(1)\big)^{k-r}\big(\Upsilon\circ \alpha_{a,\Lambda}^\ast(\tilde{\pi}^{|u_1|}X^{u_1})\big)\cdots \big(\Upsilon\circ\alpha_{a,\Lambda}^\ast(\tilde{\pi}^{|u_r|}X^{u_r})\big).$$
And extend $[\alpha_{a,\Lambda}^*]_{(k)}$ to an action on $\mathcal{S}_0^{*}(\mathcal{O}_0)$ by requiring
\begin{align*}
[\alpha_{a,\Lambda}^*]_{(k)}(e_{u_1}^*\cdots e_{u_r}^*):=&\left\{
  \begin{array}{ll}
[\alpha_{a,\Lambda}^*]_k(e_{u_1}^*\cdots e_{u_r}^*), &  \hbox{if  $r\le k$};\\
 0, & \hbox{otherwise}.
\end{array}
\right.
\end{align*}

Let $k_l$ be a sequence of integers which tends to infinity in the archimedean value and $\lim_{l\rightarrow \infty}k_l=\kappa$ $p$-adically.

\begin{lemma}\label{lem6}
As maps from $\mathcal{S}_0(\mathcal{O}_0)$ to $\mathcal{S}_0(\mathcal{O}_{0,q})$, we have that $\lim_{l\rightarrow \infty} [\alpha_{a,\Lambda}]_{(k_l)}=[\alpha_{a,\Lambda}]_{\kappa}$.
\end{lemma}

\begin{proof}
For $r\le k_l$, we have
\begin{align*}
&([\alpha_{a,\Lambda}]_{(k_l)}-[\alpha_{a,\Lambda}]_{\kappa})(e_{u_1}\cdots e_{u_r})\\
&=\big((\Upsilon(\alpha_{a,\Lambda}(1))^{k_l-r}-\Upsilon(\alpha_{a,\Lambda}(1))^{\kappa-r}\big)\Upsilon(\alpha_{a,\Lambda}(\tilde{\pi}^{|u_1|}X^{u_1}))\cdots\Upsilon(\alpha_{a,\Lambda}(\tilde{\pi}^{|u_r|}X^{u_r}))\\
&=(1-\Upsilon(\alpha_{a,\Lambda}(1))^{\kappa-k_l})(\Upsilon(\alpha_{a,\Lambda}(1)))^{k_l-r}\Upsilon(\alpha_{a,\Lambda}(\tilde{\pi}^{|u_1|}X^{u_1}))\cdots\Upsilon(\alpha_{a,\Lambda}(\tilde{\pi}^{|u_r|}X^{u_r})).
\end{align*}
Write $\kappa=k_l+\tilde\pi^{\tau(l)}\sigma_l$ with $\tau(l)\to \infty$ as $l\to \infty$ and $\sigma_l\in\mathbb Z_p$. Then we get
$$\big|([\alpha_{a,\Lambda}]_{(k_l)}-[\alpha_{a,\Lambda}]_{\kappa})(e_{u_1}\cdots e_{u_r})\big|\le |\tilde\pi|^{\tau(l)}.$$

For $r>k_l$, we have
$$([\alpha_{a,\Lambda}]_{(k_l)}-[\alpha_{a,\Lambda}]_{\kappa})(e_{u_1}\cdots e_{u_r})=-\Upsilon(\alpha_{a,\Lambda}(1))^{\kappa-r}\Upsilon(\alpha_{a,\Lambda}(\tilde{\pi}^{|u_1|}X^{u_1}))\cdots\Upsilon(\alpha_{a,\Lambda}(\tilde{\pi}^{|u_r|}X^{u_r})).$$
By Lemma \ref{lem1}, we have
$$\big|([\alpha_{a,\Lambda}]_{(k_l)}-[\alpha_{a,\Lambda}]_{\kappa})(e_{u_1}\cdots e_{u_r})\big|\le|\tilde\pi|^{(1-\frac{1}{q})|\mathbf{u}|}<|\tilde\pi|^{(1-\frac{1}{q})k_l}.$$
Hence
$$\big|([\alpha_{a,\Lambda}]_{(k_l)}-[\alpha_{a,\Lambda}]_{\kappa})\big|\le|\tilde\pi|^{\min\{\tau(l),(1-\frac{1}{q})k_l\}},$$
which means $[\alpha_{a,\Lambda}]_{(k_l)}$ tends to $[\alpha_{a,\Lambda}]_{\kappa}$ when $l$ tends to $\infty$.

\end{proof}

Similarly, for the map $\alpha_{a,\Lambda}^*$ we have the properties that $\alpha_{a,\Lambda}^*(1)=1+\eta^*$ with $|\eta^*|<1$ and
$$|\alpha_{a,\Lambda}^*( \tilde{\pi}^{-|u|}X^{-u})|\le |\tilde{\pi}|^{(p-1)|u|}.$$
We can also conclude that $\lim_{l\rightarrow \infty} [\alpha_{a,\Lambda}^*]_{(k_l)}=[\alpha_{a,\Lambda}^*]_{\kappa}$.

For positive integer $k$, define $$[\beta_a]_{(k)}:=\psi_{\Lambda}^a\circ [\alpha_{a,\Lambda}]_{(k)}$$ and
$$[\beta_a^*]_{(k)}:={\rm Pr}_1\circ [\alpha_{a,\Lambda}^*]_{(k)}\circ \Phi_{\Lambda}^a. $$
Since $\psi_{\Lambda}$ and $\Phi_{\Lambda}$ are bounded maps, it follows that
$$\lim_{l\rightarrow \infty} [\beta_a]_{(k_l)}=[\beta_a]_{\kappa}\ {\rm and}\ \lim_{l\rightarrow \infty} [\beta_a^*]_{(k_l)}=[\beta_a^*]_{\kappa}.$$

For $$\zeta=\sum_{r\in M_1,\mathbf{u}\in \mathcal{S}(M_2)}\zeta(r,\mathbf{u})\tilde{\pi}^{|r|}\Lambda^{r}e_{\mathbf{u}}\in \mathcal{S}_0^{(k)}(\mathcal{O}_0),$$
$$\zeta^*=\sum_{s\in M_1,\mathbf{v}\in \mathcal{S}(M_2)}\zeta^*(s,\mathbf{v})\tilde{\pi}^{-|s|}\Lambda^{-s}e_{\mathbf{v}}^*\in \mathcal{S}_0^{*(k)}(\mathcal{O}_0^*),$$
let $$\langle\zeta,\zeta^*\rangle _k:=\sum_{r\in M_1,\mathbf{u}\in \mathcal{S}(M_2)}\zeta(r,\mathbf{u})\zeta^*(r,\mathbf{u})\frac{1}{k!}\sum_{\sigma\in S_k}\prod_{i=1}^k\langle  e_{u_i},e_{u_{\sigma(i)}}^*\rangle ,$$
where $S_k$ denotes the symmetric group on $k$ letters. In a similar way to \cite[Lemma 4.4]{HS17}, we have

\begin{lemma}\label{lem7}
For $\zeta\in \mathcal{S}_0^{(k)}(\mathcal{O}_0)$ and $\zeta^*\in \mathcal{S}_0^{*(k)}(\mathcal{O}_0^*)$, we have that
$$\langle[\beta_a]_k(\zeta),\zeta^*\rangle _k=\langle\zeta,[\beta_a^*]_k(\zeta^*)\rangle _k.$$
\end{lemma}

As in Lemma \ref{lem3.7}, we can prove $[\beta_a]_{(k_l)}$ is also completely continuous, which means
$\det(1-[\beta_a]_{(k_l)}T)$ is well-defined, and
$$\lim_{l\rightarrow \infty}\det(1-[\beta_a]_{(k_l)}T)=\det(1-[\beta_a]_{\kappa}T).$$
From Lemma \ref{lem7} we have that
\begin{equation}\label{eq1112}
\det(1-[\beta_a^*]_{(k_l)}T)=\det(1-[\beta_a]_{(k_l)}T).
\end{equation}
That is, $\det(1-[\beta_a^*]_{(k_l)}T)$ is also well-defined. From
$\lim_{l\rightarrow \infty} [\beta_a^*]_{(k_l)}=[\beta_a^*]_{\kappa},$ we get
$$\det(1-[\beta_a^*]_{\kappa}T)=\lim_{l\rightarrow \infty}\det(1-[\beta_a^\ast]_{(k_l)}T).$$
Hence $\det(1-[\beta_a^*]_{\kappa}T)$ is also well-defined. Finally, we take limits for (\ref{eq1112}) to derive that $$\det(1-[\beta_a]_{\kappa}T)=\det(1-[\beta_a^*]_{\kappa}T).$$

Combining with Theorem \ref{thm1}, we have that
\begin{theorem}\label{thm2}
The unique unit root of the unit root L-function  $L_{unit}(f, \kappa, T)$ comes from $\det(1-[\beta_a^*]_{\kappa}T)$.
\end{theorem}

\section{The unit root formula}
Recall that $f(\Lambda,X)=\sum a_{r,u}\Lambda^{r}X^{u}$. In this section, we treat the coefficients of $f(\Lambda, X)$ as variables. Define
$$F(Y,\Lambda,X)=\sum Y_{r,u}\Lambda^{r}X^{u}.$$
The Newton polytope of $F$ at infinity with respect to $\Lambda$ and $X$ is denoted by $\Delta_{\infty}(F)\subset\mathbb{R}^{s+n}$. Let $\delta_F$ be the cone in $\mathbb{R}^{s+n}$ over $\Delta_{\infty}(F)$ and $M(F)=\delta_F\cap \mathbb{Z}^{s+n}$.
Then we have $\delta_F\subseteq\delta_1\times \delta_2$.
For $(s,v)\in \delta_1\times \delta_2$, let $$|(s, v)|:=|s|+|v|.$$

Let $B$ be the set $\{(r,u)\in {\rm supp}(f)\}$. Let $|B|$ denote the number of elements in $B$.
Let $\mathcal{K}:=R[[Y]]$ and $\mathcal{K}_0$ be the subring of $\mathcal{K}$ 
converging in the closed unit polydisk $|Y|\le 1$. That is,

$$\mathcal{K}=\big\{\zeta(Y)=\sum_{v\in \mathbb{Z}_{\ge 0}^{|B|}}c_vY^v: c_v\in R\big\}$$
and
$$\mathcal{K}_0=\big\{\zeta(Y)=\sum_{v\in \mathbb{Z}_{\ge 0}^{|B|}}c_vY^v: c_v\in R\ {\rm and}\ c_v\rightarrow 0 {\rm\ as}\ |v|\rightarrow \infty\big\}.$$
We equip $\mathcal{K}$ and $\mathcal{K}_0$ with norm $|\zeta(Y)|=\sup_{v}\{|c_v|\}$.
For $m\ge 0$, we define the following spaces:

\begin{eqnarray*}
\mathcal{O}_{0,p^m}(\mathcal{K})&=&\big\{ \sum_{r\in M_1}c_r(Y)\tilde{\pi}^{\frac{|r|}{p^m}}\Lambda^r: c_r(Y)\in \mathcal{K},c_r(Y)\rightarrow 0\ {\rm as}\ r\rightarrow \infty\big\}, \\
\mathcal{C}^*(\mathcal{O}_{0,p^m}(\mathcal{K}))&=&\big\{ \sum_{u\in M_2}\zeta_u \tilde{\pi}^{-|u|}X^{-u}:\zeta_u\in \mathcal{O}_{0,p^m}(\mathcal{K})\big\},\\
\mathcal{O}^*_{0,p^m}(\mathcal{K})&=&\big\{ \sum_{r\in M_1}c_r^*(Y)\tilde{\pi}^{-{\frac{|r|}{p^m}}}\Lambda^{-r}: c_r^*(Y)\in \mathcal{K}\big\}, \\
\mathcal{C}^*(\mathcal{O}^*_{0,p^m}(\mathcal{K}))&=&\big\{ \sum_{u\in M_2}\zeta_u^* \tilde{\pi}^{-|u|}X^{-u}:\zeta^\ast_u\in \mathcal{O}^*_{0,p^m}(\mathcal{K})\big\},\\
\mathcal{S}_0^*(\mathcal{O}_{0,p^m}(\mathcal{K}))&=&\big\{ \sum_{\mathbf{u}\in \mathcal{S}(M_2)}\zeta_{\mathbf{u}} e_{\mathbf{u}}^*:\zeta_{\mathbf{u}}\in \mathcal{O}_{0,p^m}(\mathcal{K})\big\},\\
\mathcal{S}_0^*(\mathcal{O}^*_{0,p^m}(\mathcal{K}))&=&\big\{ \sum_{\mathbf{u}\in \mathcal{S}(M_2)}\zeta_{\mathbf{u}}^* e_{\mathbf{u}}^*:\zeta^\ast_{\mathbf{u}}\in \mathcal{O}^\ast_{0,p^m}(\mathcal{K})\big\}.
\end{eqnarray*}
All these spaces with coefficients in $\mathcal K$ generalize the spaces we defined before and can be defined when replacing $\mathcal K$ by $\mathcal K_0$. We omit $p^m$ in the notations when $m=0$.  Denote by $\Upsilon$ the injection map from $\mathcal{C}^*(\mathcal{O}_{0}(\mathcal{K}))$ to
$\mathcal{S}^*(\mathcal{O}_{0}(\mathcal{K}))$ defined by
$\Upsilon(\tilde{\pi}^{-|u|}X^{-u})=e_{u}^*$ for $u\in \delta_2\setminus\{0\}$ and $\Upsilon(1)=1$.

For $m\ge 1$, let
$$H_m(Y,\Lambda,X)=\exp\big(F(Y,\Lambda,X)-F(Y^{p^m}, \Lambda^{p^m}, X^{p^m})\big).$$
As in Lemma \ref{lem01}, one verifies that
$$H_m(Y,\Lambda,X)=\sum_{(s,v)\in M(F)}B_{s,v}(Y)\tilde{\pi}^{\frac{|s|+|v|}{p^{m-1}}}\Lambda^sX^v$$
with $B_{s,v}(Y)\in \mathcal K_0$.
Define $$\alpha_{m,Y,\Lambda}^*:={\rm Pr}_2\circ H_m(Y,\Lambda,X)\circ \Phi_X^m .$$
 By the same argument as Lemma \ref{lem4} we have that $\alpha_{m,Y,\Lambda}^*$ maps $\mathcal{C}^*(\mathcal{O}_{0,p^m}(\mathcal{K}_0))$ to $\mathcal{C}^*(\mathcal{O}_{0,p^m}(\mathcal{K}_0))$. For $\kappa\in \mathbb{Z}_p$, we extend $\alpha_{m,Y,\Lambda}^*$ to a map $$[\alpha_{m,Y,\Lambda}^*]_{\kappa}:\mathcal{S}_0^*(\mathcal{O}_{0,p^m}(\mathcal{K}_0))\rightarrow \mathcal{S}_0^*(\mathcal{O}_{0,p^m}(\mathcal{K}_0))$$ by the same way as (\ref{eq0923}).
When $m=a$, a similar argument as Lemma \ref{lem5}, we have that $$[\beta_{a,Y}^*]_{\kappa}:= {\rm Pr}_1\circ [\alpha_{a,Y,\Lambda}^*]_{\kappa}\circ \Phi_{\Lambda}^a$$ is an endormorphism of $\mathcal{S}_0^*(\mathcal{O}_0^*(\mathcal{K}_0))$.

Define the projection map ${\rm Pr}_0$ on the power series ring with variables $\Lambda$ and $X$ by
$${\rm Pr}_0(\sum_{(r,u)\in \mathbb{Z}^{s}\times \mathbb{Z}^n}a_{r,u}\Lambda^rX^u):=\sum_{(r,u)\in M_0(\delta_1)\times M_0(\delta_2)}a_{r,u}\Lambda^rX^u, $$
where $M_0(\delta_1)=M_1\cap -M_1$ and $M_0(\delta_2)=M_2\cap -M_2$.

Write $${\rm Pr}_0\big(\exp \pi F(Y,\Lambda,X)\big)=\sum_{(s,v)\in M_0(\delta_1)\times M_0(\delta_2)}J_{s,v}(Y)\Lambda^sX^{v}.$$
We have $J_{s,v}(Y)\in R[[Y]]$ and $J_{0,0}(Y)\in 1+Y\cdot R[[Y]]$.
Since $J_{0,0}(Y)$ is invertible in $\mathcal{K}$, we let
$$\eta(Y,\Lambda,X):=\frac{{\rm Pr}_0\big(\exp \pi F(Y,\Lambda,X)\big)}{J_{0,0}(Y)}.$$

By a similar argument to \cite[Corollary 2.17, Corollary 2.18]{AS2}, we have $J_{s,v}(Y)/J_{0,0}(Y)\in \mathcal{K}_0$ for each $(s,v)\in M_0(\delta_1)\times M_0(\delta_2)$
 and $J_{0,0}(Y)/J_{0,0}(Y^p)\in \mathcal{K}_0$.
\begin{lemma}\label{lem1212}
For any $p$-adic integer $\kappa$, we have that $$\Upsilon\big(\eta(Y,\Lambda,X)\big)^{\kappa}\in \mathcal{S}_0^*(\mathcal{O}_0^*(\mathcal{K}_0)).$$
\end{lemma}

\begin{proof}
Write
$$\eta(Y,\Lambda,X)=1+\sum_{(r,u)\neq (0,0)}c_{r,u}(Y)\Lambda^rX^u.$$
We have proved that 
$c_{r,u}(Y)\in \mathcal{K}_0$. Since $u\in M_0(\delta_2)$, we can write
$$\eta(Y,\Lambda,X)=1+\sum_{(r,u)\neq(0,0)}c_{r,-u}(Y)\tilde{\pi}^{|u|}\Lambda^r\tilde{\pi}^{-|u|}X^{-u}.$$
Then
$$\Upsilon\big(\eta(Y,\Lambda,X)\big)=1+\sum_{(r,u)\neq(0,0)}c_{r,-u}(Y)\tilde{\pi}^{|u|}\Lambda^re_{u}^*.$$
For any $p$-adic integer $\kappa$, we have that
\begin{eqnarray*}
\Upsilon\big(\eta(Y,\Lambda,X)\big)^{\kappa}=\sum_{l=0}^{\infty}\binom{\kappa}{l}\bigg(\sum_{(r,u)\neq(0,0)}c_{r,-u}(Y)\tilde{\pi}^{|u|}\Lambda^re_{u}^* \bigg)^l
=\sum_{r\in M_0(\delta_1),\mathbf{u}\in \mathcal{S}(\mathcal{M})}D_{r,\mathbf{u}}(Y)\Lambda^re_{\mathbf{u}}^*,
\end{eqnarray*}
where $D_{r,\mathbf{u}}(Y)\in \mathcal{K}_0$ is well-defined. Since $r\in M_0(\delta_1)$, we rewrite $\Upsilon\big(\eta(Y,\Lambda,X)\big)^{\kappa}$ as
$$\Upsilon\big(\eta(Y,\Lambda,X)\big)^{\kappa}
=\sum_{r\in M_0(\delta_1),\mathbf{u}\in \mathcal{S}(\mathcal{M})}D_{-r,\mathbf{u}}(Y)\tilde{\pi}^{|r|}(\tilde{\pi}^{-|r|}\Lambda^{-r})e_{\mathbf{u}}^*.$$
We see that $\Upsilon\big(\eta(Y,\Lambda,X)\big)^{\kappa}\in \mathcal{S}_0^*(\mathcal{O}_0^*(\mathcal{K}_0)).$
This finishes the proof of Lemma \ref{lem1212}.

\end{proof}

 Let $\mathcal{F}(Y):=J_{0,0}(Y)/J_{0,0}(Y^p)$ and $\mathcal{F}_m(Y):=\prod_{i=0}^{m-1}\mathcal{F}(Y^{p^i})$ for $m\ge 1$.

\begin{theorem}\label{thm01}
Let $\hat{a}_{r,u}$ be the Teichm\"{u}ller lift of $a_{r,u}$ in ${\mathbb Q}_q$, $\hat{a}=(\hat{a}_{r,u})$ and $\mathcal{F}(\hat{a})=\mathcal{F}(Y)|_{Y=\hat{a}}$. Then
$\mathcal{F}_a(\hat{a})^{\kappa}=\big(\mathcal{F}(\hat{a})\mathcal{F}(\hat{a}^p)\cdots \mathcal{F}(\hat{a}^{p^{a-1}})\big)^{\kappa}$ is the unique unit root of $L_{unit}(f,\kappa,T)^{(-1)^{s+1}}$
at $Y_{r,u}=a_{r,u}$.
\end{theorem}
\begin{proof}

We have that
 \begin{align*}
 \alpha_{1,Y,\Lambda}^*\big(\eta(Y^p,\Lambda^p,X)\big)&={\rm Pr}_2\Big( H(Y,\Lambda,X)\circ \frac{{\rm Pr}_0\big(\exp \pi F(Y^p,\Lambda^p,X^p)\big)}{J_{0,0}(Y^p)}\Big)\\
 &={\rm Pr}_2\Big( H(Y,\Lambda,X)\big(\frac{\exp \pi F(Y^p,\Lambda^p,X^p)}{J_{0,0}(Y^p)} +\omega'(Y,\Lambda,X)+\epsilon(Y,\Lambda,X)\big)\Big)\\
 &={\rm Pr}_2\Big(\frac{\exp \pi F(Y,\Lambda,X)}{J_{0,0}(Y^)}+H(Y,\Lambda,X)\big(\omega'(Y,\Lambda,X)+\epsilon(Y,\Lambda,X)\big) \Big),
 \end{align*}
where $\Lambda^rX^u$ appearing in $\omega'(Y,\Lambda,X)$ has $r$ in $\delta_1\setminus M_0(\delta_1)$ and $\Lambda^rX^u$ appearing in $\epsilon(Y,\Lambda,X)$ has $u$ in $\delta_2\setminus M_0(\delta_2)$.

Let ${\rm Pr}_{2,0}$ be the projection of exponents of $X$ onto $M_0(\delta_2)$, that is
$${\rm Pr}_{2,0}(\sum_{r,u}a_{r,u}\Lambda^rX^u):=\sum_{u\in M_0(\delta_2)}a_{r,u}\Lambda^rX^u. $$
Recall that ${\rm Pr}_2$ is the projection of exponents of $X$ onto $-M_2$, and all exponents of $X$ in $H(Y,\Lambda,X)$ and $\epsilon(Y,\Lambda,X)$ lie in $M_2$. Hence we have
$${\rm Pr}_2\big(H(Y,\Lambda,X)\epsilon(Y,\Lambda,X)\big)={\rm Pr}_{2,0}\big(H(Y,\Lambda,X)\epsilon(Y,\Lambda,X)\big)={\rm Pr}_{2,0}\big(H(Y,\Lambda,X)\big)\cdot{\rm Pr}_{2,0}\big(\epsilon(Y,\Lambda,X)\big)=0.$$
By $\mathcal{F}(Y)=J_{0,0}(Y)/J_{0,0}(Y^p)$, we have that
\begin{align*}
\alpha_{1,Y,\Lambda}^*\big(\eta(Y^p,\Lambda^p,X)\big)&=\mathcal{F}(Y)\Big({\rm Pr}_2\Big(\frac{\exp \pi F(Y,\Lambda,X)}{J_{0,0}(Y)}\Big)+\frac{1}{\mathcal{F}(Y)}{\rm Pr}_2(H(Y,\Lambda,X)\cdot \omega'(Y,\Lambda,X))\Big)\\
 &=\mathcal{F}(Y)\big(\eta(Y,\Lambda,X)+\omega^*(Y,\Lambda,X)\big),
 \end{align*}
where $\Lambda^rX^u$ appearing in $\omega^*(Y,\Lambda,X)$ has $r$ in $\delta_1\setminus M_0(\delta_1)$.
One verifies $$\alpha_{a,Y,\Lambda}^*=\alpha_{1,Y,\Lambda}^\ast \circ \alpha_{1,Y^p,\Lambda^p}^\ast\circ\cdots \circ\alpha_{1,Y^{p^{a-1}},\Lambda^{p^{a-1}}}^\ast.$$
 By repeating the above process, we have that
\begin{equation}\label{map}
\alpha_{a,Y,\Lambda}^*\big(\eta(Y^q,\Lambda^q,X)\big)=\mathcal{F}_a(Y)\big(\eta(Y,\Lambda,X)+\omega(Y,\Lambda,X)\big),
\end{equation}
where $\Lambda^rX^u$ appearing in $\omega(Y,\Lambda,X)$ has $r$ lie in $\delta_1\setminus M_0(\delta_1)$. For any $l\ge 1$, one can see that $\Lambda^rX^u$ appearing $\eta^{\kappa-l}\omega^l$ also satisfies that $r\in \delta_1\setminus M_0(\delta_1)$.
Then
\begin{equation}\label{maps}
{\rm Pr}_1\big(\Upsilon(\eta(Y,\Lambda,X))+\Upsilon(\omega(Y,\Lambda,X))\big)^{\kappa}= {\rm Pr}_1\Big(\sum_{l=0}^{\infty}\binom{\kappa}{l}\Upsilon\big(\eta(Y,\Lambda,X)\big)^{\kappa-l}\Upsilon\big(\omega(Y,\Lambda,X)\big)^l\Big)=\Upsilon\big(\eta(Y,\Lambda,X)\big)^{\kappa}.
\end{equation}
Write $\eta(Y,\Lambda,X)=1+h(Y,\Lambda,X)$.

It follows from  Lemma \ref{lem1212} that
$\Upsilon\big(\eta(Y,\Lambda,X)\big )^{\kappa}\in \mathcal{S}_0^*(\mathcal{O}_0^*(\mathcal{K}_0))$. Combining with  (\ref{map}) and (\ref{maps}), we have

 \begin{align*}
 [\beta_{a,Y}^*]_{\kappa}\Big(\Upsilon\big(\eta(Y^q,\Lambda,X)\big)^{\kappa}\Big)&={\rm Pr}_1\circ [\alpha_{a,Y,\Lambda}^*]_{\kappa}\circ \Phi_{\Lambda}^a\Big(\Upsilon\big(\eta(Y^q,\Lambda,X)\big)^{\kappa}\Big)\\
 &={\rm Pr}_1\circ [\alpha_{a,Y,\Lambda}^*]_{\kappa}\Big(\Upsilon\big(\eta(Y^q,\Lambda^q,X)\big)^{\kappa}\Big)\\
 &={\rm Pr}_1\circ [\alpha_{a,Y,\Lambda}^*]_{\kappa}\Big(\sum_{l=0}^{\infty}\binom{\kappa}{l}\Upsilon\big(h(Y^q,\Lambda^q,X)\big)^l \Big)\\
 &={\rm Pr}_1\circ \Big(\sum_{l=0}^{\infty}\binom{\kappa}{l}\Upsilon\big(\alpha_{a,Y,\Lambda}^*(1)\big)^{\kappa-l}\Upsilon\big(\alpha_{a,Y,\Lambda}^*(h(Y^q,\Lambda^q,X))\big)^l \Big)\\
 &={\rm Pr}_1\circ \Big(\Upsilon\big(\alpha_{a,Y,\Lambda}^*(1)\big)+\Upsilon\big(\alpha_{a,Y,\Lambda}^*(h(Y^q,\Lambda^q,X))\big) \Big)^{\kappa}\\
 &={\rm Pr}_1\circ\Upsilon\Big( \alpha_{a,Y,\Lambda}^*\big(\eta(Y^q,\Lambda^q,X)\big)\Big)^{\kappa}\\
 &={\rm Pr}_1\circ\Upsilon \Big(\mathcal{F}_a(Y)\big(\eta(Y,\Lambda,X)+\omega(Y,\Lambda,X)\big)\Big)^{\kappa}\\
 &=\mathcal{F}_a(Y)^{\kappa}\cdot{\rm Pr}_1 \circ\Upsilon \big(\eta(Y,\Lambda,X)+\omega(Y,\Lambda,X)\big)^{\kappa}\\
 &=\mathcal{F}_a(Y)^{\kappa}\cdot \Upsilon\big(\eta(Y,\Lambda,X)\big)^{\kappa}.
 \end{align*}

Since the coefficients of $\eta(Y,\Lambda,X)$ with respect to $\Lambda$ and $X$ belong to $\mathcal{K}_0$, and $ \mathcal{F}_a(Y)\in \mathcal{K}_0$, we can specialize the above equality by setting $Y_{r,u}=\hat{a}_{r,u}$.
We then obtain
$$[\beta_a^*]_{\kappa}\Big(\Upsilon\big(\eta(\hat{a},\Lambda,X)\big)^{\kappa}\Big)=\mathcal{F}_a(\hat{a})^{\kappa}\cdot \Upsilon\big(\eta(\hat{a},\Lambda,X)\big)^{\kappa}.$$
We observe that $\mathcal{F}_a(\hat{a})^{\kappa}=\big(\mathcal{F}(\hat{a})\mathcal{F}(\hat{a}^p)\cdots \mathcal{F}(\hat{a}^{p^{a-1}})\big)^{\kappa}$ is the characteristic value of $[\beta_a^*]_{\kappa}$.
Therefore, by Theorem \ref{thm2}, we conclude that
the unique unit root of the unit root L-function  $L_{unit}(f, \kappa, T)$ is $\mathcal{F}_a(\hat{a})^{\kappa}$.
This completes the proof of Theorem \ref{thm01}.
\end{proof}

\begin{center}
{\sc Acknowledgements}
\end{center}
H. Zhang is supported by National Natural Science Foundation of China (Grant No. 12171261, Grant No. 12301012), and by the Natural Science Foundation of Jiangsu Province (Grant No. BK20230802).
L.P. Yang is supported by National Natural Science Foundation of China (Grant No. 12201078).

\end{document}